\documentclass[12pt]{amsart}
\usepackage{geometry}
\usepackage{graphicx,color}
\usepackage{amssymb, mathrsfs, amsfonts, amsmath}
\usepackage{latexsym}
\usepackage{parskip}
\usepackage{url}

\newcommand\ris[5]{\raisebox{#1mm}{\hspace{#2mm}\includegraphics[width=#3mm]{#4.eps}\hspace{#5mm}}}

\newtheorem{theorem}{Theorem}[section]
\newtheorem{lemma}[theorem]{Lemma}
\newtheorem{corollary}[theorem]{Corollary}
\newtheorem{proposition}[theorem]{Proposition}
\newtheorem{definition}[theorem]{Definition}

\newtheorem{remark}[theorem]{Remark}

\newcommand{\R}{\mathbb R}
\newcommand{\Q}{\mathbb Q}

\begin{document}

\title{Lipschitz geometry of surface germs in $\R^4$: metric knots}
\author[]{Lev Birbrair*}\thanks{*Research supported under CNPq 302655/2014-0 grant and by Capes-Cofecub}
\address{Departamento de Matem\'atica, Universidade Federal do Cear\'a
(UFC), Campus do Pici, Bloco 914, Cep. 60455-760. Fortaleza-Ce,
Brasil} \email{lev.birbrair@gmail.com}

\author[]{Michael Brandenbursky**}\thanks{**Research partially supported by Humboldt foundation}
\address{Department of Mathematics, Ben Gurion University of the Negev,
Beer Sheva, Israel}\email{brandens@bgu.ac.il}

\author[]{Andrei Gabrielov$\dagger$}\thanks{$\dagger$ Research supported by the NSF grant DMS-1665115}
\address{Department of Mathematics, Purdue University,
West Lafayette, IN 47907, USA}\email{gabriea@purdue.edu}

\date{\today}

\begin{abstract}
 {A link at the origin of an isolated singularity of a two-dimensional semialgebraic
surface in $\R^4$ is a topological knot (or link) in $S^3$. We study the connection
between the ambient Lipschitz geometry of semialgebraic surface germs in $\R^4$
and knot theory. Namely, for any knot $K$, we construct a surface $X_K$
in $\R^4$ such that: the link at the origin of $X_{K}$ is a trivial knot; the germs $X_K$
are outer bi-Lipschitz equivalent for all $K$; two germs $X_{K}$ and $X_{K'}$
are ambient semialgebraic bi-Lipschitz equivalent only if the knots $K$ and $K'$ are isotopic.
We show that the Jones polynomial can be used to recognize ambient bi-Lipschitz
non-equivalent surface germs in $\R^4$,
even when they are topologically trivial and outer bi-Lipschitz equivalent.}
\end{abstract}
\maketitle

\section{Introduction}
We study the difference between the outer and ambient bi-Lipschitz equivalence of semialgebraic
 surface germs at the origin in $\R^4$. Two surface germs are outer bi-Lipschitz
 equivalent if they are bi-Lipschitz equivalent as abstract metric spaces with the outer metric $d(x,y)=\|x-y\|$.
 Ambient bi-Lipschitz equivalence means that there exists a germ of a bi-Lipschitz,
orientation preserving, homeomorphism of the ambient space mapping one of them to the other one.
Note that in Singularity Theory the homeomorphism is not required to be orientation preserving.
We add this condition to be consistent with the isotopy equivalence relation in Knot Theory.
Also, to avoid confusion between the Singularity Theory and Knot Theory notions of the link,
we always write ``the link at the origin'' speaking of the link of a surface germ.

 If a surface germ in $\R^4$ with a connected link at the origin has an isolated singularity then its link is a knot in $S^3$.
 The results of \cite{BG} show that
 ambient equivalence is different from outer equivalence even when there are no topological
 obstructions. This phenomenon is called ``metric knots.''
 We consider the following question: How different are these equivalence relations?
 In the previous paper \cite{BG} we show that, for any given ambient topological type of a surface germ, one can find infinitely many equivalence classes with respect to ambient  bi-Lipschitz equivalence. In this paper we start by showing that the
 question becomes nontrivial even when ``there is no topology,'' i.e., for the germs with unknotted links at the origin.
 Universality Theorem (Theorem \ref{universality} below) implies that the ambient bi-Lipschitz classification
 in this case ``contains all of Knot Theory.''
 Namely, for any knot $K$, one can construct a germ of a surface $X_K$ in $\R^4$ such that:\newline
 1. The link at the origin of $X_{K}$ is a trivial knot;\newline
 2. The germs $X_K$ are outer bi-Lipschitz equivalent for all $K$;\newline
 3. Two germs $X_{K}$ and $X_{K'}$ are ambient semialgebraic bi-Lipschitz equivalent only if the knots $K$ and $K'$ are isotopic.

The second theorem (Theorem \ref{twist} below) states that,
for each germ $X_K$ in Universality Theorem, there are infinitely many semialgebraic
surfaces $X_{K,i}$ satisfying Universality Theorem, such that $X_{K,i}$ and $X_{K,j}$ are semialgebraic ambient bi-Lipschitz equivalent only if $i=j$.

The proofs are based on the following results of Sampaio \cite{S}  and Valette \cite{V}.

\begin{theorem}\cite[Theorem 2.2]{S}\label{Sampaio}
If $(X,0)$ and $(Y,0)$ are ambient semialgebraic bi-Lipschitz equivalent semialgebraic germs, then their tangent cones $C_0(X)$ and $C_0(Y)$ are ambient semialgebraic bi-Lipschitz equivalent.
\end{theorem}
\begin{theorem}\cite[Corollary 0.2]{V}\label{Valette}
If two semialgebraic germs  $(X,0)$ and $(Y,0)$ are semialgebraic bi-Lipschitz homeomorphic, then there is a semialgebraic bi-Lipschitz homeomorphism $h:(X,0)\to(Y,0)$ preserving the distance to the origin.
\end{theorem}

In Section \ref{Metric Knots} we define $(\beta_1,\beta_2)$-bridges and the saddle move,
closely related to the broken bridge construction in \cite{BG}.
A \emph{one-bridge surface germ} is a surface germ containing a single $(\beta_1,\beta_2)$-bridge and metrically conical outside it.
The saddle move relates the metric problem of ambient semialgebraic
bi-Lipschitz equivalence of two one-bridge surface germs in $\R^4$ with the
topological problem of isotopy of two knots in $S^3$ corresponding to the links at the origin of the surfaces
obtained from these one-bridge surface germs by the saddle moves (see Definition \ref{saddle}).
That is why topological knot invariants, such as the Jones polynomial,
yield metric knot invariants, which can be used to recognize ambient semialgebraic bi-Lipschitz non-equivalence of surface germs.

Although one-bridge surface germs are the simplest examples of not normally embedded surfaces, they  have rather non-trivial ambient Lipschitz geometry.
Another version of Universality Theorem (Theorem \ref{twoknots} below) states that, for any two knots $K$ and $L$,
one can construct a one-bridge surface germ $X_{KL}$ such that:\newline
1. The link at the origin of $X_{KL}$ is isotopic to $L$;\newline
2. For any knots $K$ and $L$, all surface germs $X_{KL}$ are outer bi-Lipschitz equivalent;\newline
3. Surface germs $X_{{K_1}L}$ and $X_{{K_2}L}$ are ambient semialgebraic bi-Lipschitz
equivalent only if the knots $K_1$ and $K_2$ are isotopic.

In Section \ref{invariants} we consider the Jones polynomial of the link at the origin $L=L_{S(X)}$ of a surface germ $S(X)$ obtained from a one-bridge surface germ $X$ by the saddle move (see Definition \ref{saddle}). Since the isotopy class of $L$ is an ambient semialgebraic Lipschitz invariant, its Jones polynomial becomes an ambient Lipschitz invariant of $X$.
If $X=X'_{K,i}$ is a ``twisted'' surface constructed in \cite{BG} (see also Theorem \ref{twist}) and $K$ is a trivial knot, then $L$ is a torus link. Its Jones polynomial is computed completely (see Corollary \ref{specialcase} and Remark \ref{last-remark}) and determines the number $i$ of twists.
This shows that Jones polynomial can be used to prove ambient bi-Lipschitz non-equivalence of metric knots.

If we do not suppose the surface germ to be a one-bridge surface germ, we obtain a stronger version of Universality Theorem (Theorem \ref{twoknots-wedge} below). It states that, for any two knots $K$ and $L$, and any two rational numbers $\alpha > 1$ and $\beta > 1$,
  one can construct a surface germ  $X^{\alpha\beta}_{KL}$ such that:\newline
1. The link at the origin of $X^{\alpha\beta}_{KL}$ is isotopic to $L$;\newline
2. For a fixed knot $K$, the tangent link of $X^{\alpha\beta}_{KL}$ (i.e., the intersection of the tangent cone with the unit sphere)  is isotopic to $K$; \newline
3. All surface germs $X^{\alpha\beta}_{KL}$ are outer bi-Lipschitz equivalent for fixed $\alpha$ and $\beta$.

All sets, functions and maps in this paper are assumed to be real semialgebraic.
We use semialgebraic bi-Lipschitz equivalence, because we refer to the theorem of Valette \cite{V}.
Our results are also true for subanalytic bi-Lipschitz equivalence of subanalytic surface germs, and
we expect them to remain true in any polynomially bounded o-minimal structure over $\R$.

\section{Definitions and Notations}
We consider germs at the origin of semialgebraic surfaces (two-dimensional semialgebraic sets) in $\R^4$.

\begin{definition}
\emph{A surface $X$ can be considered as a metric space, equipped with either the outer metric $d(x,y)=\|x-y\|$
or the inner metric $d_i(x,y)$ defined as the minimal length of a path in $X$ connecting $x$ and $y$.
A germ $X$ is \emph{normally embedded} if its inner and outer metrics are equivalent.}
\end{definition}

\begin{definition}\label{equivalences}
\emph{Two germs of semialgebraic sets $(X,0)$ and $(Y,0)$ are \emph{outer bi-Lipschitz equivalent}
if there exists a homeomorphism $H: (X,0)\rightarrow (Y,0)$ bi-Lipschitz with respect to the outer metric. The germs are \emph{ semialgebraic outer bi-Lipschitz equivalent} if the map $H$ can be chosen to be semialgebraic.
The germs are \emph{ambient bi-Lipschitz equivalent} if there exists an orientation preserving bi-Lipschitz homeomorphism
$\widetilde{H}: (\R^4,0) \rightarrow (\R^4,0)$, such that $\widetilde{H}(X)=Y$. The germs are \emph{ semialgebraic ambient bi-Lipschitz equivalent} if the map $\widetilde{H}$ can be chosen to be semialgebraic.}
\end{definition}
\begin{definition}\label{link}
\emph{The \emph{link at the origin} $L_X$ of a germ $X$ is the equivalence class of the sets $X\cap S^3_{0,\varepsilon}$ for small positive $\varepsilon$ with respect to the ambient bi-Lipschitz equivalence. The \emph{tangent link} of $X$ is the link at the origin of the tangent cone of $X$.}
\end{definition}
\begin{remark}\label{well-defined}
\emph{By the finiteness theorems of Mostowski, Parusinski and Valette (see \cite{Mostowski}, \cite{Parusinski} and \cite{valette2005Lip}) the link at the origin is well defined.
We write ``the link at the origin'' speaking of this notion of the link from Singularity Theory, reserving
the word ``link'' for the notion of the link in Knot Theory.
If $X$ has an isolated singularity at the origin then each connected component of $L_X$ is a knot in $S^3$.}
\end{remark}

\begin{definition}\label{conical}
\emph{A semialgebraic germ $(X,0) \subset \R^n$ is called \emph{outer metrically conical} if there exists a germ of a bi-Lipschitz homeomorphism $H:(X,0)\to C(L_X)$, where $C(L_X)$ is a straight cone over $L_X$. The map $H$ is called a \emph{conification map}.
A germ $(X,0)$ is called \emph{ambient metrically conical} if there exists a germ of a bi-Lipschitz homeomorphism $\widetilde{H}:\R^n\to \R^n$, such that $\widetilde{H}(X)=C(L_X)$. The map $\widetilde{H}$ is also called a \emph{conification map}.
A germ $(X,0)$ is called \emph{outer (ambient) semialgebraic metrically conical} if a corresponding conification map can be chosen to be semialgebraic.}
\end{definition}

\begin{remark}\label{conical-remark}
\emph{Notice that the definition makes sense for semialgebraic germs of any dimension, not only for surface germs.}
\end{remark}

\begin{definition}\label{def:arc}\normalfont
An \emph{arc} in a semialgebraic germ $(X,0)$ is a germ of a semialgebraic embedding
$\gamma:[0,\epsilon)\rightarrow X$ such that $\gamma(0)=0$. Unless otherwise specified,
 arcs are parameterized by the distance to the origin, i.e., $\|\gamma(t)\|=t$.
 We identify an arc with its image in $X$.
\end{definition}

\begin{definition}\label{ord}
\emph{Let $f\not\equiv 0$ be (a germ at the origin of) a semialgebraic function defined on an arc $\gamma$.
The \emph{order} $\alpha$ of $f$ on $\gamma$ (notation $\alpha=ord_\gamma f$) is the value $\alpha\in\Q$
such that $f(\gamma(t))=ct^{\alpha}+o(t^{\alpha})$ as $t\to 0$, where $c\ne 0$.
If $f\equiv 0$ on $\gamma$, we set $ord_\gamma f=\infty$.}
\end{definition}

For any two arcs $\gamma$ and $\gamma'$ in $X$ one can define two orders of contact: inner and outer.

\begin{definition}\label{contactorder}
\emph{The \emph{outer order of contact} $tord(\gamma,\gamma')$ is defined as  $ord_{\gamma}f$, where $f(t)=\|\gamma(t)-\gamma'(t)\|$.
The \emph{inner order of contact} $itord(\gamma,\gamma')$ is defined as $ord_{\gamma}g$, where $g(t)=d_p(\gamma(t),\gamma'(t))$.
Here $d_{p}$ is a definable pancake metric (see \cite{BM}) equivalent to the inner metric.  These two orders of contact are rational
numbers  such that $1\le itord(\gamma,\gamma')\le tord(\gamma,\gamma')$.}
\end{definition}

\begin{definition}\label{wedge-bridge}
\normalfont Let $\beta>1$ be a rational number. Consider the space $\R^3$ with coordinates $(x,y,z)$.
For a fixed $t\ge 0$, let $Z_t=\{|x|\le t,\;|y|\le t\}$ be a square in the $(x,y)$-plane $\{z=t\}$ and let $Z=\bigcup_{t\ge 0}Z_t$.
Let $W^+_t$ be the subset of $Z_t$
bounded by the line segment $I^+_t=\{|x|\le t,\;y=t\}$ and the union $J^+_t$ of the two line segments connecting the endpoints
of $I^+_t$ with the point $(0,t^\beta)$.
Let $W^-_t=\{(x,y):(x,-y)\in W^+_t\}$ and $J^-_t=\{(x,y):(x,-y)\in J^+_t\}$.
Let $W_t=W^+_t\cup W^-_t$ (shaded area in Figure~\ref{fig:wt}a) and let $W={\bigcup}_{t\ge 0}\, W_t\subset \R^3$.
A \emph{$\beta$-bridge} is the surface germ $B_\beta={\bigcup}_{t\ge 0}\,J^+_t\cup J^-_t$ .
Note that the tangent cone of $W$ is the set $\{|x|\le |y|\le z\}$
and the tangent cone of $B_\beta$ is the surface germ $\{|x|=|y|\le z\}$.
\end{definition}

\begin{figure}
\centering
\includegraphics[width=5.5in]{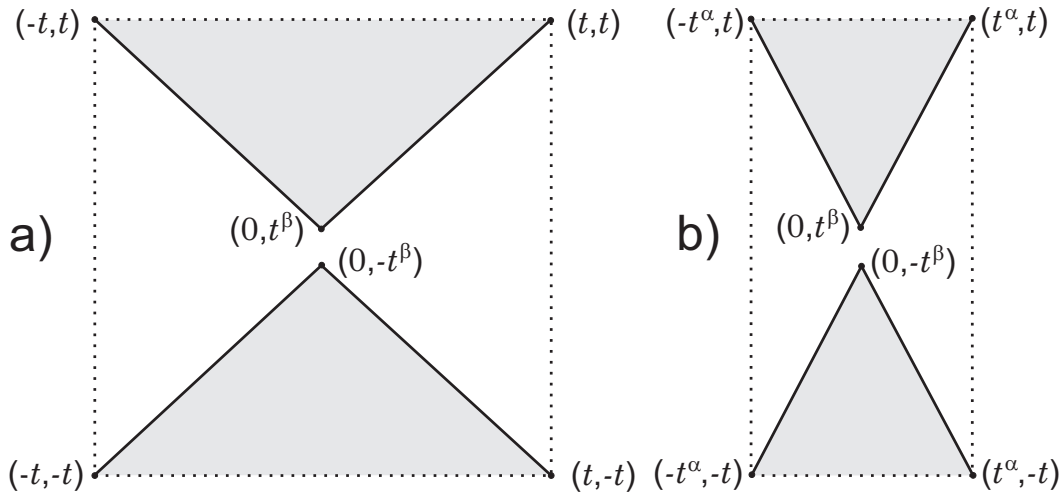}
\caption{a) The set $W_t$. b) The set $W^{\alpha\beta}_t$.}\label{fig:wt}
\end{figure}
\begin{figure}
\centering
\includegraphics[width=4in]{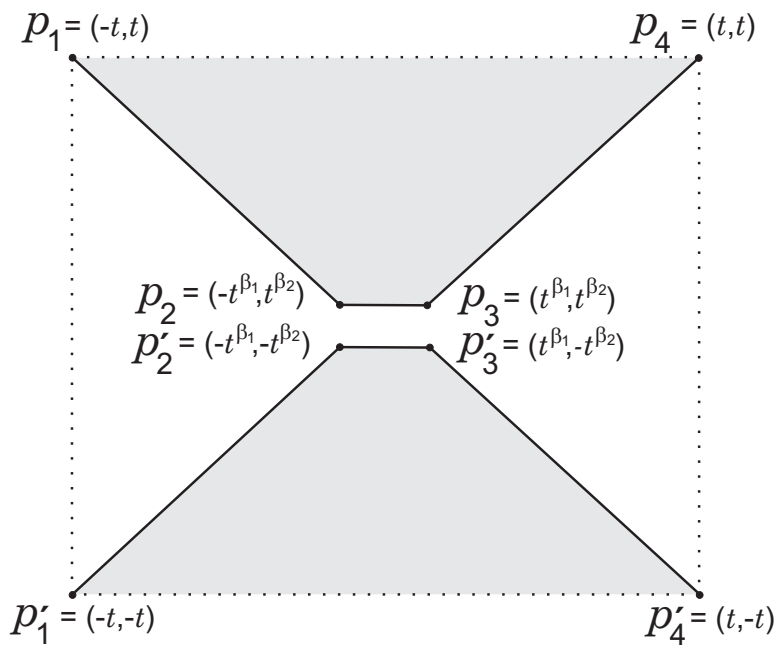}
\caption{The set $U_t$.}\label{fig:pic4}
\end{figure}

\begin{definition}\label{bridge}
\normalfont Let $1<\beta_1\le\beta_2$ be two rational numbers.
For a fixed $t\ge 0$, let $Z_t=\{|x|\le t,\;|y|\le t,\;z=t\}$ and $Z=\bigcup_{t\ge 0}Z_t$  be as in Definition \ref{wedge-bridge}.
In the $xy$-plane $\{z=t\}$ consider the points (see Figure~\ref{fig:pic4})
$$p_1(t)=(-t, t),\; p_2(t)=(-t^{\beta_1}, t^{\beta_2}),\;p_3(t)= (t^{\beta_1}, t^{\beta_2}) ,\;  p_4(t)= (t, t),$$
$$p'_1(t)= (-t, -t),\;  p'_2(t)= (-t^{\beta_1}, -t^{\beta_2}),\;   p'_3(t)= (t^{\beta_1},-t^{\beta_2}),\;  p'_4(t)= (t, -t).$$
Let us connect the points $p_1(t),\,p_2(t),\,p_3(t),\,p_4(t)$ by three line segments, and define $\bar{J}^+_t$ as the union of these three segments.
Let $U^+_t \subset Z_t$ be the convex hull of $\bar{J}^+_t$. Let $P^+_t$ be the segment connecting the points $p_2(t)$ and $p_3(t)$.
Similarly, let $\bar{J}^-_t$ be the union of segments connecting the points $p_1(t),\,p'_2(t),\,p'_3(t),\,p'_4(t)$, and let $U^-_t$ be the convex hull of $\bar{J}^-_t$ and $P^-_t$ be the segment connecting $p'_2(t)$ with $p'_3(t)$. Let $P_t=P^+_t\cup P^-_t$ and let $P=\bigcup_{t\ge 0} P_t \subset \R^3$.
Let $U_t=U^+_t\cup U^-_t$ (shaded area in Figure~\ref{fig:pic4}), and let  $U=\bigcup_{t\ge 0} U_t \subset \R^3$.
A \emph{$(\beta_1,\beta_2)$-bridge} is the surface germ $B_{\beta_1 \beta_2}=\bigcup_{t\ge 0}\bar{J}_t$,
where $\bar{J}_t=\bar{J}^+_t\cup\bar{J}^-_t$.

Note that the set $U$ has the same tangent cone at the origin as $W$, while the tangent cone at the origin of $P$ is the positive $z$-axis.
Note also that, for $\beta_1=\beta_2=\beta$, the $(\beta,\beta)$-bridge is outer bi-Lipschitz equivalent to the $\beta$-bridge.
 \end {definition}

\begin{definition}\label{one-bridge}
\normalfont Let $X$ be a semialgebraic surface germ in $\R^4$ with the link at the origin homeomorphic to a circle in $S^3$.
We say that $X$ is a \emph{one-bridge surface germ} if\newline
1. There exists a semialgebraic bi-Lipschitz $C^1$ embedding $\Theta:Z\to\R^4$ such that $\Theta(B_{\beta_1 \beta_2})=X\cap \Theta(Z)$.\newline
2. The union  $X\cup \Theta(Z)$ is normally embedded in $\R^4$ and ambient semialgebraic metrically conical: there exist a semialgebraic
bi-Lipschitz homeomorphism $\widetilde{H}:\R^4\to\R^4$, such that $\widetilde{H}(X\cup\Theta(Z))$ is a straight cone.
 \end{definition}

\begin{definition}\label{wedge}
\normalfont Let $\alpha>1$ and $\beta>1$ be rational numbers. Consider the space $\R^3$ with coordinates $(x, y, z)$.
For a fixed $t\ge 0$, let $Z^\alpha_t=\{|x|\le t^\alpha,\;|y|\le t\}$ be a rectangle
in the $(x,y)$-plane $\{z=t\}$. Let $W_t^{\alpha+}$ be the subset of the rectangle $Z^\alpha_t$
bounded by the line segment $I_t^{\alpha+}=\{|x|\le t^\alpha,\;y=t\}$ and the union $J^{\alpha+}_t$
of the line segments connecting the endpoints of $I_t^{\alpha+}$ with the point $(0,t^\beta)$.
Let $W^{\alpha-}_t=\{(x,y): (x,-y)\in W^{\alpha+}_t\}$ and $J^{\alpha-}_t=\{(x,y):(x,-y)\in J^{\alpha+}_t\}$.
Let $W^\alpha_t=W_t^{\alpha+}\cup W_t^{\alpha-}$ (shaded areas in Figure~\ref{fig:wt}b) and let $W^{\alpha}={\bigcup}_{t\ge 0}\, W^{\alpha}_t\subset \R^3$.
An \emph{$(\alpha,\beta)$-wedge} is the surface germ $E^{\alpha\beta}=\bigcup_{t\ge 0}\,J^\alpha_t$, where $J^\alpha_t=J^{\alpha+}_t\cup J^{\alpha-}_t$.

Note that the tangent cone at the origin of $W^\alpha$ is the set $\{(x,y,z): x=0; |y|\le z\}$.
\end{definition}

\begin{remark}\label{diagram}
\emph{We define a \emph{link diagram} in the same way as it is done in Knot Theory, choosing a generic projection of the topological link to some 2-dimensional plane in $\R^3$ (see \cite{Ka} for details). Two diagrams are equivalent if they can be related by a finite sequence of Reidemeister moves.}
\end{remark}

The following result is a special case of the finiteness theorem of Hardt (see \cite{Hardt}).

\begin{theorem}\label{finiteness}
Let $X$ be a semialgebraic surface germ. Then, for small $t>0$ and for any plane $\R^2$, such that the projections of the links $X\cap S_t$ are generic, the diagrams of the links $X\cap S_t$ are equivalent.
\end{theorem}

\begin{definition}\label{characteristic-band}\normalfont
 Let $F_K \subset S^{3}$ be a smooth semialgebraic embedded surface diffeomorphic to $S^{1} \times [-1,1]$,
such that the two components  $\widetilde{K}$ and $\widetilde{K}'$ of the boundary $\partial F_K$ of $F_K$ are isotopic to the same knot $K$
and the linking number (see \cite{Milnor}) of the components $\widetilde{K}$ and $\widetilde{K}'$ is zero.
The surface $F_K$ is called a \emph{characteristic band} of the knot $K$.
Let $\widetilde{Y}_K$ and $\widetilde{X}_K$ be the cones over $F_K$ and $\partial F_K$, respectively.
These cones are called \emph{characteristic cones} of the knot $K$.
\end{definition}

\begin{definition}\label{slice}\normalfont
Let $(\rho,l)$, where $\rho \in S^1$ and $l\in [-1,1]$, be coordinates in $F_K$.
Let $\xi=(\rho_0,0)$ be an interior point of $F_K$.
We define a \emph{slice} $S_K=\{(\rho,l)\in F_K,\;|\rho-\rho_0|\le\epsilon\}$.
\end{definition}

\begin{definition}\label{horn}\normalfont
Let $\beta> 1$ be a rational number.
The standard \emph{$\beta$-horn} in $\R^4$ is the set $C_\beta = \{(x,y,z,t)\in \R^4 \mid t\ge 0, \; x^2+y^2+z^2 = t^{2\beta}\}$.
The standard \emph{$\beta$-horn like neighborhood} of the positive $t$-axis is the set
$V_\beta = \{(x,y,z,t)\in \R^4 \mid t\ge 0, \; x^2+y^2+z^2 \le t^{2\beta}\}$.

If $\beta=1$ then $C_1=\{t\ge 0,\;x^2+y^2+z^2=t^2\}$ is a cone and $V_1=\{t\ge 0,\;x^2+y^2+z^2 \le t^2\}$ is a conical neighborhood of the positive $t$-axis.

The standard \emph{$\beta$-hornification} $\Xi_\beta:V_1\to V_\beta$ is defined as $\Xi_\beta(x,y,z,t)=(xt^\beta,yt^\beta,zt^\beta,t)$.

For an arc $\gamma\subset\R^4$, a conical neighborhood of $\gamma$ is the image $V_1(\gamma)$ of a semialgebraic bi-Lipschitz map $\Phi:V_1\to\R^4$ such that $\gamma$ is the image of the positive $t$-axis.
A $\beta$-horn like neighborhood of $\gamma$ is $V_\beta(\gamma)=\Phi(V_\beta)$,
and a $\beta$-hornification to $\gamma$ is the map $\Psi_\beta:V_1(\gamma)\to V_\beta(\gamma)$ defined as $\Psi_\beta=\Phi\,\Xi_\beta\,\Phi^{-1}$ (see Figure \ref{horn-curve}).
We may assume, by Valette's theorem, that $\Psi_\beta$ preserves the distance to the origin.
For a subset $S$ of $V_1(\gamma)$, the set $\Psi_\beta(S)$ is called a \emph{$\beta$-hornification} of $S$ to $\gamma$.
\end{definition}

\begin{figure}
\centering
\includegraphics[width=5.5in]{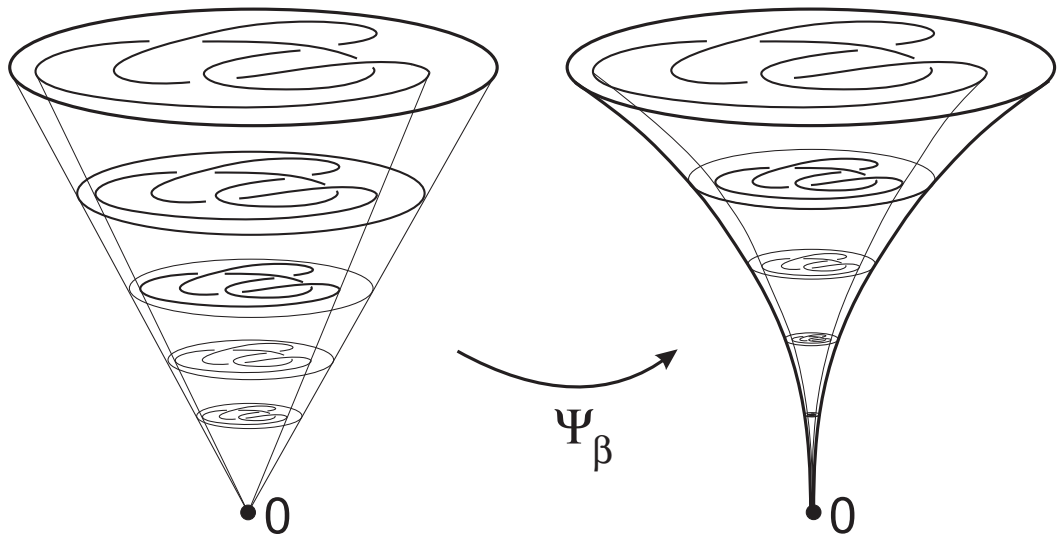}
\caption{Hornification of the cone over a knot}\label{horn-curve}
\end{figure}

\section{Metric Knots}\label{Metric Knots}.

\begin{figure}
\centering
\includegraphics[width=5in]{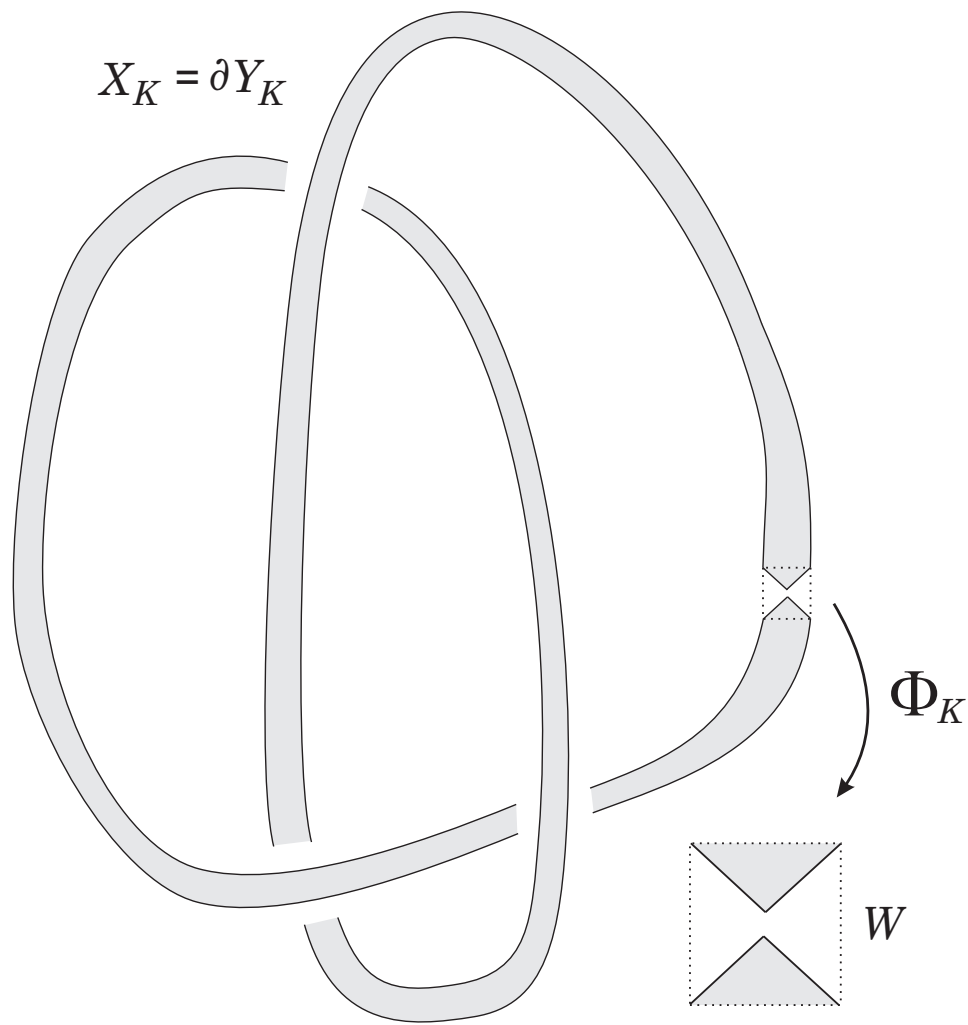}
\caption{The links of the sets $X_K=\partial Y_K$ and $W$ in the proof of Theorem \ref{universality}.}\label{fig:xk}
\end{figure}

\begin{theorem}[Universality Theorem]\label{universality}
Let $K \subset {S}^3$ be a knot. Then one can associate to $K$ a semialgebraic one-bridge surface germ $(X_K,0)$ in $\R^4$ so that the
following holds:\newline
$1) \ $ The link at the origin of each germ $X_K$ is a trivial knot;\newline
$2) \ $ All germs $X_K$ are outer bi-Lipschitz equivalent;\newline
$3) \ $ Two germs $X_{K_1}$ and $X_{K_2}$ are ambient semialgebraic bi-Lipschitz equivalent only if the knots $K_1$ and $K_2$ are isotopic.
\end{theorem}

\begin{proof}
Let $F_K \subset S^3$ be a characteristic band of the knot $K$, and let  $\widetilde{Y}_K$ and $\widetilde{X}_K$ be the
corresponding characteristic cones (see Definition \ref{characteristic-band}). Let $S_K\subset F_K$ be a slice (see Definition \ref{slice}).
Let $\varphi_K:S_K\to Z_1$, where $Z_1$ is the set $Z_t$ in Definition \ref{wedge-bridge} with $t=1$, be a semialgebraic bi-Lipschitz  homeomorphism $(\rho,l)\mapsto((\rho-\rho_0)/\epsilon,l)$.
Let $M_K=\{t\sigma: t\ge 0,\;\sigma\in S_K\}\subset\R^4$ be the cone over $S_K$.
We define a mapping $\Phi_K: M_K \to Z \subset\R^3$ as the corresponding mapping of the cones:
\begin{equation}\label{slice-cone}
\Phi_K(t\sigma)=t\varphi_K(\sigma)\;\text{for}\;\sigma\in S_K.
\end{equation}
 Note that $\Phi_K$ is a bi-Lipschitz homeomorphism. Let $W\subset\R^3$ be the set in Definition \ref{wedge-bridge}, and let
\begin{equation}\label{VKYKXK}
V_K = \Phi^{-1}_K(W),\quad Y_K = \left( \widetilde{Y}_K\setminus M_K \right) \cup V_K,\quad X_K=\partial Y_K.
\end{equation}

Then $X_K$ is a one-bridge surface germ, part of the surface germ $\widetilde{X}_K$ inside $M_K$
being replaced by a $\beta$-bridge $B_\beta$ (see Figure \ref{fig:xk}).
Let us show that $X_K$ satisfies the conditions of Theorem \ref{universality}. \vspace{0,1cm}

$1) \ $ The link at the origin of $X_K$ is a trivial knot, because it bounds the closure of $F_K \setminus S_K$ homeomorphic to a disk.

$2) \ $ Let $K_1$ and $K_2$ be any two knots. Let $\Psi: \widetilde{Y}_{K_1} \to \widetilde{Y}_{K_2}$ be a semialgebraic bi-Lipschitz map  sending each point  $(\rho,l,t)\in\widetilde{Y}_{K_1}$ to the point $(\rho,l,t)\in\widetilde{Y}_{K_2}$.
By definition $\Psi(M_{K_1})=M_{K_2}$. By the definition of the maps $\Phi_{K_1}$ and $\Phi_{K_2}$ (see (\ref{slice-cone})) we have $\Psi(Y_{K_1})=Y_{K_2}$ and $\Psi(X_{K_1}) = X_{K_2}$.

$3) \ $ Note that, for any knot $K$, the link of the tangent cone $C_0 X_K$ of the set $X_K$ is the union of two
knots isotopic to $K$, with a single common point.
Thus if $K_1$ and $K_2$ are not isotopic, then the tangent cones $C_0 X_{K_1}$ and $C_0 X_{K_2}$ are not ambient
topologically equivalent. This contradicts Sampaio's
theorem  \cite{S} (see also Theorem \ref{Sampaio}) which implies that tangent cones of ambient Lipschitz equivalent semialgebraic sets are ambient Lipschitz equivalent.
In our case, the links of the tangent cones are not even ambient topologically equivalent.

This concludes the proof of Theorem \ref{universality}.
\end{proof}

\begin{definition}\label{band-bridge} \normalfont
A surface germ $X_K$ obtained by the above construction is called a \emph{band-bridge surface germ} corresponding to the knot $K$ and a $\beta$-bridge (or a $(\beta_1,\beta_2)$-bridge as in the proof of Theorem \ref{twist} below).
\end{definition}

\begin{definition}\label{saddle}
\normalfont
Consider the $(\beta_1,\beta_2)$-bridge $B_{\beta_1\beta_2}=\bigcup_{t\ge 0}\bar{J}_t$
(see Definition \ref{bridge}). The set $\bar{J}_t$ has two components $\bar{J}^+_t$ and $\bar{J}^-_t$,
each of them consisting of three line segments connecting the points $p_1(t),\,p_2(t),\,p_3(t),\,p_4(t)$ and $p'_1(t),\,p'_2(t),\,p'_3(t),\,p'_4(t)$, respectively (see Figure \ref{fig:pic4}).
Let $\hat{J}_t$ be the set obtained by replacing the line segments $[p_2(t),p_3(t)]$ and $[p'_2(t),p'_3(t)]$ in $\bar{J}_t$ with the line segments $[p_2(t),p'_2(t)]$ and $[p_3(t),p'_3(t)]$ (see Figure~\ref{pic6}a).
Let $S_{\beta_1\beta_2}=\bigcup_{t\ge 0}\hat{J}_t$.
Let $X$ be a one-bridge surface germ (see Definition \ref{one-bridge}). Replacing $\widetilde{B}=\Theta(B_{\beta_1\beta_2})\subset X$  with $\widetilde{S}=\Theta(S_{\beta_1\beta_2})$, we obtain a new surface germ $S(X)$. This defines the \emph{saddle move} operation applied to $X$.
\end{definition}

\begin{figure}
\centering
\includegraphics[width=4in]{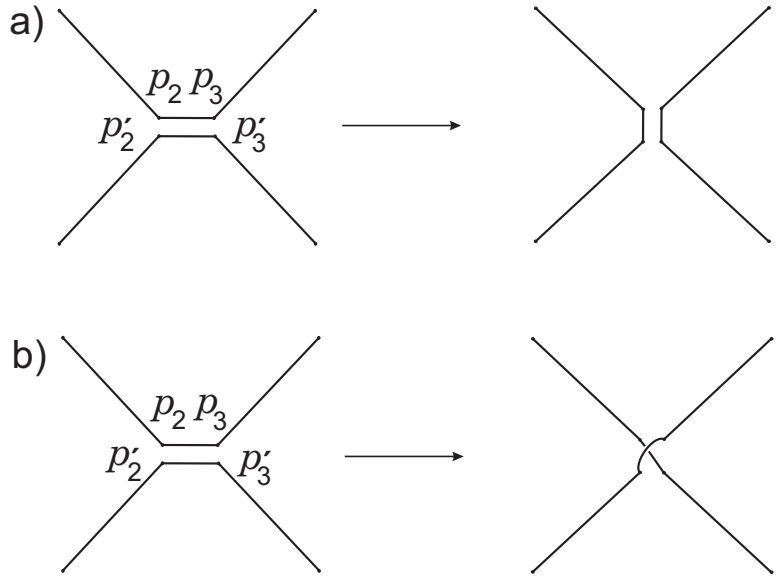}
\caption {a) The saddle move. b) The crossing move.}\label{pic6}
\end{figure}

\begin{lemma}\label{bridgelemma}
Let $X_1$ and $X_2$ be semialgebraic ambient bi-Lipschitz equivalent one-bridge surface germs.
Then the surface germs $S(X_1)$ and $S(X_2)$, obtained by the saddle move applied to $X_1$ and $X_2$, are ambient topologically equivalent, the links at the origin $L_{S(X_1)}$ and $L_{S(X_2)}$  are isotopic as topological links in $S^3$, and the diagrams of the links $L_{S(X_1)}$ and $L_{S(X_2)}$ are equivalent.
\end{lemma}

\begin{proof} Let $Z\subset\R^3$ be as in Definitions \ref{wedge-bridge} and \ref{bridge}.
Let $\Theta_1:Z\to\R^4$ and $\Theta_2:Z\to\R^4$ be bi-Lipschitz embeddings such that $\widetilde{B}_1=\Theta_1(B_{\beta_1\beta_2})\subset X_1$ and $\widetilde{B}_2=\Theta_2(B_{\beta_1\beta_2})\subset X_2$.
Since $X_1$ and $X_2$ are one-bridge surfaces, we can suppose that $X_1\cup \Theta_1(Z)$ and $X_2\cup \Theta_2(Z)$ are straight cones over their links.
Let $H:\R^4 \to \R^4$ be a bi-Lipschitz homeomorphism isotopic to identity such that $H(X_1)=X_2$. By Valette's Theorem \cite{V} (see also Theorem \ref{Valette}) we may suppose that $H$ preserves the distance to the origin, and that the maps $\Theta_1$ and $\Theta_2$ send each section $Z_t$ of $Z$ to the sphere $S_t$ of radius $t$ centered at the origin.

Let $\widetilde{P}_1=\Theta_1(P)$ and $\widetilde{P}_2=\Theta_2(P)$, where $P=\bigcup_{t\ge 0} P_t\subset B_{\beta_1\beta_2}$ (see Definition \ref{bridge})
and let $\widetilde{P}_1(t)=\Theta_1(P_t)=\widetilde{P}_1\cap S_t$ and $\widetilde{P}_2(t)=\Theta_2(P_t)=\widetilde{P}_2\cap S_t$.
Since the tangent cone $C_0P$ of $P$ is the positive $z$-axis,
the tangent cones $C_0\widetilde{P}_1$ and $C_0\widetilde{P}_2$ of $\widetilde{P}_1$ and $\widetilde{P}_2$ are rays in $\R^4$.
For a small positive $\epsilon$, let $N_t\subset S_t$ be a ball of radius $\epsilon t$ centered at the point $C_0\widetilde{P}_2\cap S_t$,
and let $N=\bigcup_{t\ge 0} N_t$ be a conical $\epsilon$-neighbourhood of $C_0\widetilde{P}_2$.
Note that $\widetilde{P}_2\subset N\cap X_2\subset\widetilde{B}_2$ for small $\epsilon>0$.

Let $p_2(t),\,p'_2(t),\,p_3(t),\,p'_3(t)$ be the boundary points of $P_t$ (see Definition \ref{bridge}).
Let $q_2(t)=\Theta_1(p_2(t))$, $q'_2(t)=\Theta_1(p'_2(t))$, $q_3(t)=\Theta_1(p_3(t))$, $q'_3(t)=\Theta_1(p'_3(t))$  be the
boundary points of $\widetilde{P}_1(t)$,
and let $v_2(t)=\Theta_2(p_2(t))$, $v'_2(t)=\Theta_2(p'_2(t))$, $v_3(t)=\Theta_2(p_3(t))$, $v'_3(t)=\Theta_2(p'_3(t))$ be
the boundary points of $\widetilde{P}_2(t)$.
Then $\tilde q_2(t)=H(q_2(t))$, $\tilde q'_2(t)=H(q'_2(t))$, $\tilde q_3(t)=H(q_3(t))$, $\tilde q'_3(t)=H(q'_3(t))$ are
the boundary points of $H(\widetilde{P}_1(t))$.

The saddle move operation applied to $X_1$ replaces $\widetilde{P}_1$ with $Q=\bigcup_{t\ge 0} Q_t$,
where $Q_t=Q_2(t)\cup Q_3(t)$, $Q_2(t)=\Theta_1([p_2(t),p'_2(t)])$, $Q_3(t)=\Theta_1([p_3(t),p'_3(t)])$.
The saddle move operation applied to $X_2$ replaces $\widetilde{P}_2$ with
$V=\bigcup_{t\ge 0} V_t$, where $V_t=V_2(t)\cup V_3(t)$, $V_2(t)=\Theta_2([p_2(t),p'_2(t)])$, $V_3(t)=\Theta_2([p_3(t),p'_3(t)])$.
Let $\widetilde{Q}=H(Q),\,\widetilde{Q}_2=H(Q_2),\,\widetilde{Q}_3=H(Q_3)$.
Note that the boundary points $\tilde q_2(t),\,\tilde q'_2(t),\,\tilde q_3(t),\,\tilde q'_3(t)$ of $\widetilde{Q}_t$ are the same as the boundary points of $H(\widetilde{P}_1(t))$,
and the boundary points $v_2(t),\,v'_2(t)$, $v_3(t),\,v'_3(t)$ of $V_t$ are the same as the boundary points of $\widetilde{P}_2(t)$.
In particular, all these points belong to the bridge $\widetilde{B}_2$ of $X_2$, and to the $\epsilon t$-ball $N_t$ (see Figure \ref{pic7}).

Note that $tord(\widetilde{Q}_2,\widetilde{Q}_3)=tord(V_2,V_3)=\beta_1$ and $tord(\widetilde{Q}_2,V_2)=tord(\widetilde{Q}_3,V_3)=\beta_2$. Consider
$diam(V_2(t)),\, diam(V_3(t)),\, diam(\widetilde{Q}_2(t)),\,diam(\widetilde{Q}_3(t))$ as functions of $t$. Note that the order of all of these functions at the origin is $\beta_2$. Let $N_2$ be the family of balls $N_{2,t}$ on $S_t$ centered at $v_2(t)$ with the radius $t^{\tilde{\beta}}$ for $\tilde{\beta} \in (\beta_1,\beta_2)$, and let $N_3$ be the family of balls $N_{3,t}$ on $S_t$ centered at $v_3(t)$ with the radius $t^{\tilde{\beta}}$. Clearly $N_{2,t}\cap N_{3,t} = \emptyset$ and also $\widetilde{Q}_2(t)\subset N_{2,t},\, V_2(t)\subset N_{2,t},\, \widetilde{Q}_3(t)\subset N_{3,t},\,V_2(t)\subset N_{3,t}$.

Since $\widetilde{Q}_2(t),\,\widetilde{Q}_3(t),\,V_2(t),\,V_3(t)$ are homeomorphic to segments, there exists a homeomorphism $\bar{H}_2:N_1 \to N_1$ isotopic to identity, such that : \newline
1. $\bar{H}_2$ maps the sections $z=t$ to the sections $z=t$.\newline
2. $\bar{H}_2$ is identity on the boundary of $N_1$.\newline
3. $\bar{H}_2(H(\widetilde{Q}_2))=V_2$.\newline
4. The bridge $\widetilde{B}_2=H(\widetilde{B}_1)$ is invariant under $\bar{H}_2$.\newline
Similarly, there exists a homeomorphism $\bar{H}_3:N_2 \to N_2$ isotopic to identity, such that : \newline
1. $\bar{H}_3$ maps the sections $z=t$ to the sections $z=t$.\newline
2. $\bar{H}_3$ is identity on the boundary of $N_2$.\newline
3. $\bar{H}_3(H(\widetilde{Q}_3))=V_3$.\newline
4. The bridge $\widetilde{B}_2=H(\widetilde{B}_1)$ is invariant under $\bar{H}_3$.\newline
Then we define a homeomorphism $H':R^4\to R^4$ to be equal to $H$ outside $H^{-1}(N_1\cup N_2)$, to $\bar{H}_2*H$ on $H^{-1}(N_1)$ and to $\bar{H}_3*H$ on $H^{-1}(N_2)$, thus $H'(S(X_1))=S(X_2)$.
This proves that $S(X_1)$ and $S(X_2)$ are ambient topologically equivalent, and the links at the origin $L_{S(X_1)}$ and $L_{S(X_2)}$  are isotopic as topological links.
\end{proof}

\begin{figure}
\centering
\includegraphics[width=5in]{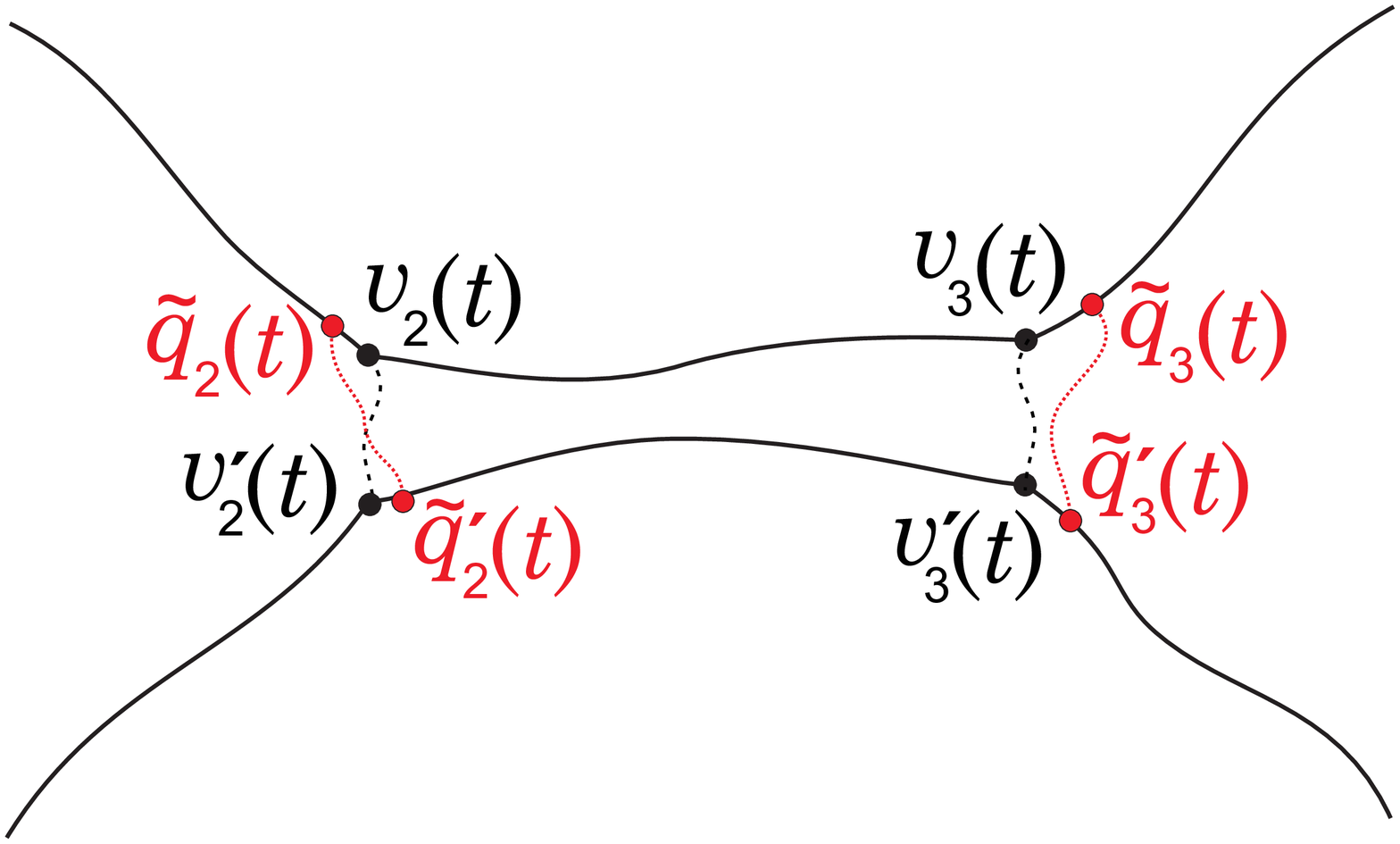}
\caption {The images by the map $H$ in the proof of Lemma \ref{bridgelemma}.}\label{pic7}
\end{figure}
\begin{theorem}\label{twist}
For any knot $K\subset {S}^3$ and all integers $i\ge 0$, there exist semialgebraic surface germs $(X'_{K,i},0)$ in $\R^4$ such that:

$1) \ $ The tangent cones at the origin of all $X'_{K,i}$ are topologically equivalent  to the cone over two knots isotopic to
$K$ with a single common point.

$2) \ $ All $X'_{K,i}$ are outer bi-Lipschitz equivalent.

$3) \ $ $ X'_{K,i}$ and $X'_{K,j}$ are semialgebraic ambient bi-Lipschitz equivalent only when $i=j$.
\end{theorem}

\begin{proof}
Consider a characteristic band $F_K\subset S^3$, a slice $S_K\subset F_K$, and characteristic cones $\widetilde{Y}_K$ and $\widetilde{X}_K$ (see Definitions \ref{characteristic-band} and \ref{slice}).
We construct a band-bridge surface germ with a $(\beta_1,\beta_2)$-bridge corresponding to $K$ as follows.
Let $M_K=\{t\sigma: t\ge 0,\;\sigma\in S_K\}\subset\R^4$ be the cone over $S_K$ (as in Theorem \ref{universality}).
Let $\Phi_K: M_K\to Z\subset\R^3$ be the map defined in (\ref{slice-cone}):
\begin{equation}\label{slice-slice-cone}
\Phi_K(t\sigma)=t\varphi_K(\sigma)\;\text{for}\;\sigma\in S_K.
\end{equation}
Note that $\Phi_K$ is a bi-Lipschitz homeomorphism. For $1<\beta_1\le\beta_2$, let $U\subset\R^3$ be the set in Definition \ref{bridge}.
We define
\begin{equation}\label{new-map}
V'_{K,0} = \Phi^{-1}_K(U),\quad Y'_{K,0} = \left( \widetilde{Y}_K,\setminus M_K \right) \cup V'_{K,0}\quad X'_{K,0}=\partial Y'_{K,0}.
\end{equation}
The set $Y'_{K,0}$ is obtained by replacing the set $W$ (see Definition \ref{wedge-bridge}) with the set $U$ in construction
of the set $X_K$ in the proof of Theorem \ref{universality}. Let $X'_{K,0}=\partial Y'_{K,0}$ be its boundary (see Figure~\ref{fig:pic5}).
This construction replaces a $\beta$-bridge in Theorem \ref{universality} by a $(\beta_1,\beta_2)$-bridge.
In particular, the one-bridge surface germ $(X'_{K,0},0)$ satisfies conditions of Theorem \ref{universality}.

\begin{figure}
\centering
\includegraphics[width=4in]{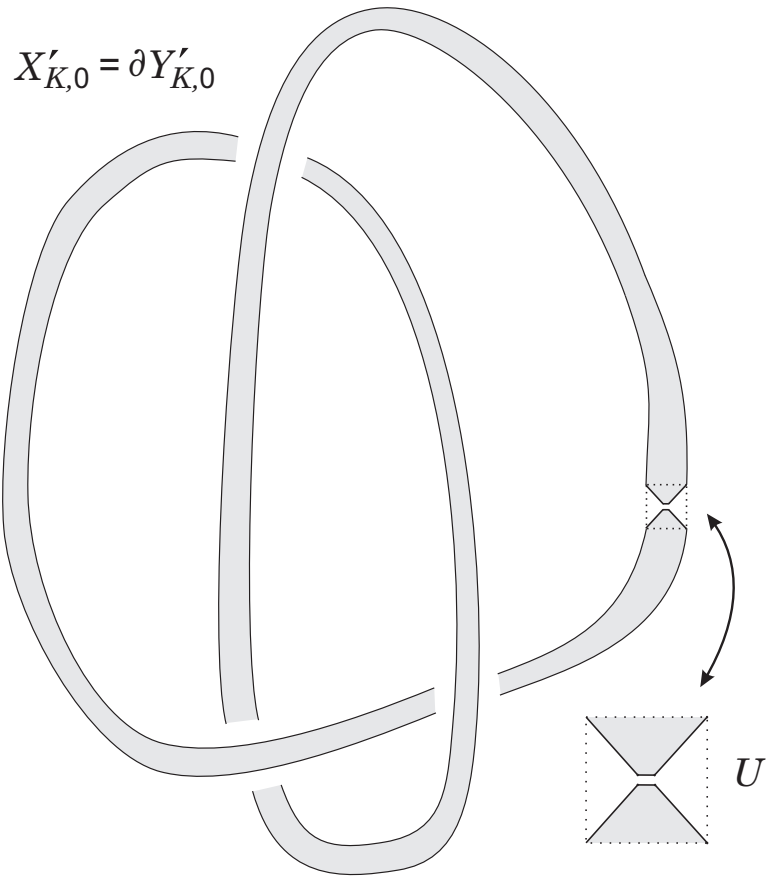}
\caption { The links of the surface $X'_{K,0}=\partial Y'_{K,0}$ and $U$ in the proof of Theorem \ref{twist}.}\label{fig:pic5}
\end{figure}

\begin{figure}
\centering
\includegraphics[width=5in]{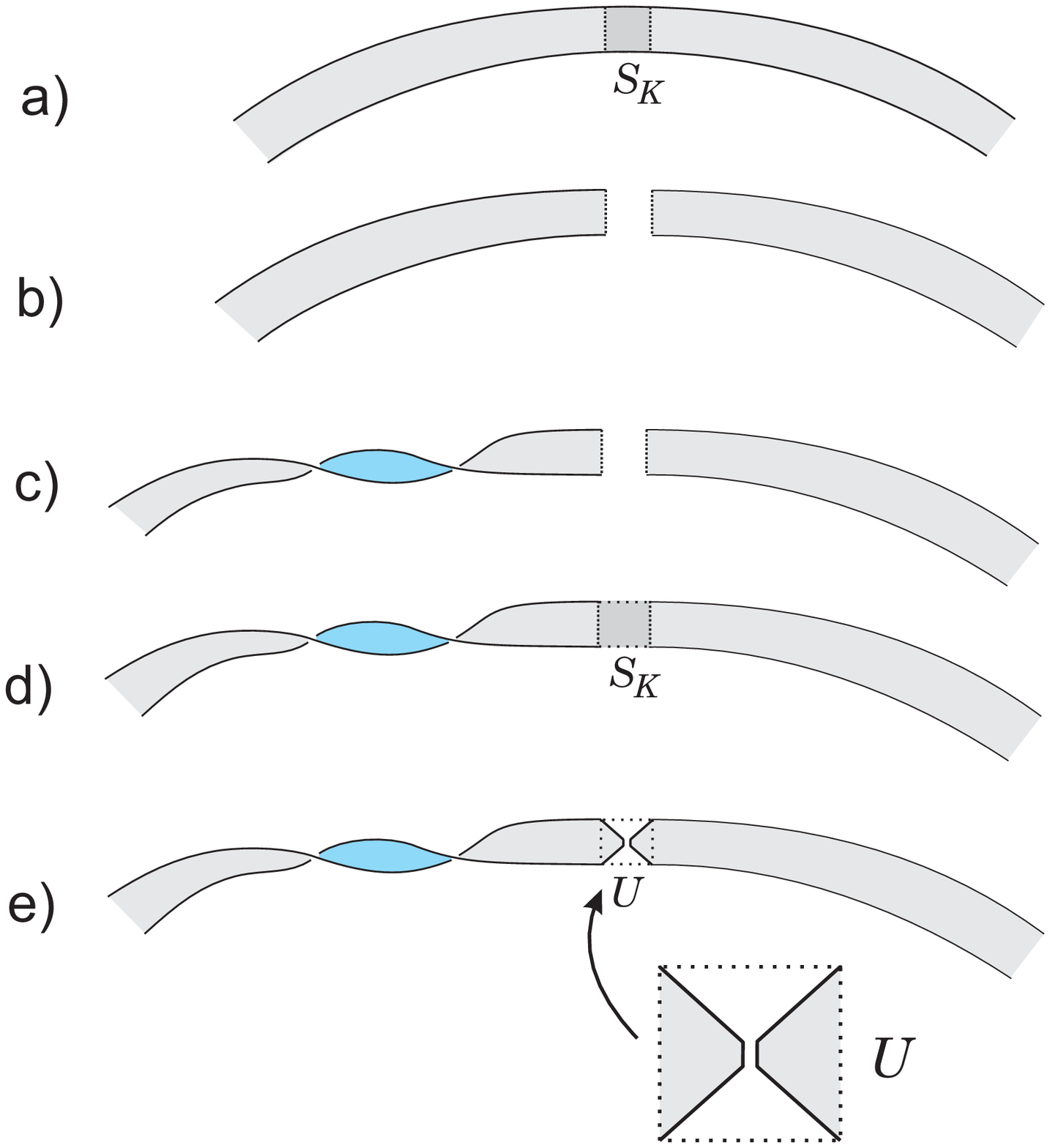}
\caption{Cut and twist in the proof of Theorem \ref{twist}.}\label{pic3}
\end{figure}

Let now $F'_{K,i}$ be the set obtained by removing the slice $S_K$ from $F_K$, making $i$
complete twists and adding $S_K$ back (see Figure~\ref{pic3}a-d).
Let $Y'_{K,i}$ be the set obtained from the cone over $F'_{K,i}$ by replacing the set $M_K$ (the cone over $S_K$) with the set $U$ (see Figure~\ref{pic3}e)
and let $X'_{K,i}=\partial Y'_{K,i}$ be its boundary.
The same arguments as in the proof of Theorem \ref{universality} show that the link of $X'_{K,i}$ is a
trivial knot and the tangent cone of $X'_{K,i}$
is a cone over the union of two knots isotopic to $K$, pinched at one point.

We are going to prove that $X'_{K,i}$ and $X'_{K,j}$ are not semialgebraic ambient bi-Lipschitz equivalent if $i\ne j$.
The result of the saddle move applied to each of these surface germs is a surface germ such that its tangent link is the union of two
copies of the knot $K$, with the linking number of the two copies being twice the number of complete twists.
 Thus the links $S(X'_{K,j})$ and $S(X'_{K,i})$ are not isotopic when $i\ne j$.
 It follows from Lemma \ref{bridgelemma} that surface germs $X'_{K,i}$ and $X'_{K,j}$ are ambient semialgebraic bi-Lipschitz equivalent when $i\ne j$.

Note that the topology of the tangent link of $X'_{K,i}$ does not depend on $i$. The tangent link is formed
by two copies of $K$ pinched at one point.
\end{proof}

\begin{remark}\label{rmk:saddle}\normalfont
Let $X'_{K,i}$ be the surface germ constructed in the proof of Theorem \ref{twist}.
Then the link at the origin of the surface germ $S(X'_{K,i})$, obtained from $X'_{K,i}$ by the saddle move, is a subset of $F'_{K,i}$ isotopic to $\partial F'_{K,i}$.
\end{remark}

\begin{proposition}\label{twist-twist}
Let $X'_{K,i}$ be a surface germ constructed in Theorem \ref{twist}, and let $S(X'_{K,i})$ be the surface germ
obtained by a saddle move applied to $X'_{K,i}$.
If $K$ is a trivial knot, then the link at the origin of $S(X'_{K,i})$ is a torus link.
\end{proposition}

\begin{proof} For a small $\epsilon>0$, the boundary of the $\epsilon$-neighbourhood of $K$ is an unknotted two-dimensional torus $T_K\subset S^3$.
One can define coordinates $(\phi,\psi)$ on $T_K$, where $\phi\in K,\,\psi\in S^1$, so that the curves $\tilde K=\{\phi\in K, \,\psi=0\}$
and $\tilde K'=\{\phi\in K, \,\psi=\pi\}$ have the linking number zero. Then $F_K=\{\phi\in K, \,0\le\psi\le\pi\}\subset T_K$ is
a characteristic band of the knot $K$ (see Definition \ref{characteristic-band}) bounded by the curves $\tilde K$ and $\tilde K'$.
If $X'_{K,0}$ is the surface germ constructed in Theorem \ref{twist}, then the
link at the origin of $S(X'_{K,0})$, isotopic to the union of $\tilde K$ and $\tilde K'$, is a trivial torus link.

The surgery for constructing a surface germ $X'_{K,i}$ in Theorem \ref{twist} (see Figure \ref{pic3}) corresponds to the choice of a coordinate system $(\phi,\psi_i)$ on $T_K$ such that the band $F_{K,i}=\{\phi\in K, \,0\le\psi_i\le\pi\}\subset T_K$ is bounded by the curves $\tilde K_i=\{\phi\in K,\,\psi_i=0\}$ and $\tilde K'_i=\{\phi\in K,\,\psi_i=\pi\}$ with the linking number $2i$.
Since the link at the origin of $S(X'_{K,i})$ is isotopic to the union of $\tilde K_i$ and $\tilde K'_i$ (see Remark \ref{rmk:saddle}) it is a torus link.
\end{proof}

\begin{proposition}\label{saddle-saddle}
Let $X_1$ and $X_2$ be two one-bridge surface germs. If the germs are ambient bi-Lipschitz equivalent,
then the links of the origin  $L_{S(X_1)}$ and $L_{S(X_2)}$ are isotopic.
\end{proposition}

\begin{remark}\label{saddle-diagrams}\normalfont
Saddle move on the level of knot diagrams is described as follows: $\left\langle\ris{-4}{-1.1}{10}{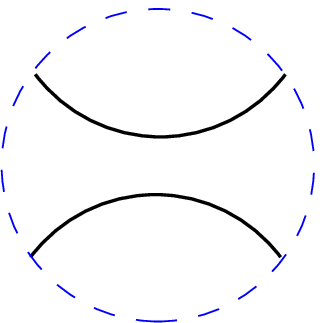}{-1.1}\right\rangle$ is replaced by $\left\langle\ris{-4}{-1.1}{10}{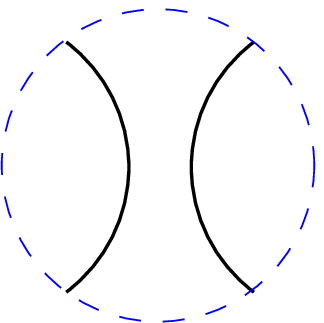}{-1.1}\right\rangle$.

\end{remark}

Here we are going to define the crossing move, that will be useful for further calculations.

\begin{definition}\label{crossin}
\normalfont We proceed in a similar way to the definition of the saddle move. Consider the subset $B$ of a one-bridge surface $X$ outer bi-Lipschitz equivalent to a $(\beta_1,\beta_2)$-bridge
$B_{\beta_1\beta_2}=\bigcup_{t\ge 0}\bar{J}_t$ (see Definition \ref{bridge}). The set $\bar{J}_t$ has two components $\bar{J}^+_t$ and $\bar{J}^-_t$,
consisting of three line segments connecting the points $p_1(t),\,p_2(t),\,p_3(t),\,p_4(t)$ and $p'_1(t),\,p'_2(t),\,p'_3(t),\,p'_4(t)$, respectively, in the plane  $\{z=t, w=0\}$ (see Figure \ref{fig:pic4}). Let us embed this set to $\R^4$ with coordinates $(x,y,z,w)$.
Replacing the line segments $[p_2(t),p_3(t)]$ and $[p'_2(t),p'_3(t)]$ with the line segment
$[p_2(t),p'_3(t)]$ and a circle arc in the half-space $\{w\ge 0\}$ with the
ends at $p'_2(t)$ and $p_3(t)$, orthogonal to the plane $\{w=0\}$ (see Figure~\ref{pic6}b), we replace the set $\bar{J}_t$ with the set $\hat{J}_t$.
Let $\widehat{B}_{\beta_1\beta_2}=\bigcup_{t\ge 0}\hat{J}_t$.
Note that the surface germs $B_{\beta_1\beta_2}$ and $\widehat{B}_{\beta_1\beta_2}$ have the same boundary arcs.
Replacing the subset $B$ of $X$ with the subset $\widehat{B}$ outer bi-Lipschitz equivalent to $\widehat{B}_{\beta_1\beta_2}$,
so that $B$ and $\widehat{B}$ have the same boundary arcs, we get a new surface germ $C(X)$. 
This defines a \emph{crossing move} operation applied to $X$.
\end{definition}

\begin{remark}\label{crossing}\normalfont Crossing move on the level of knot diagrams is described as follows:
$\left\langle\ris{-4}{-1.1}{10}{smoothing-rev}{-1.1}\right\rangle$ is replaced by $\left\langle\ris{-4}{-1.1}{10}{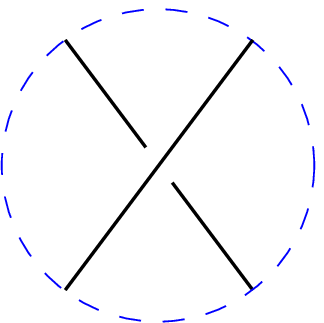}{-1.1}\right\rangle$.

\end{remark}

\begin{remark} \label{crossing-remark}\normalfont
One can show that, for the fixed orientation on $L_X$, the isotopy class of the resulting knot or link is an ambient bi-Lipschitz invariant. 
However, in what follows we do not need this result.
\end{remark}

The next statement is a modification of the Universality Theorem.

\begin{theorem}\label{twoknots}

For any two knots $K$ and $L$, there exists a germ of a semialgebraic one-bridge surface germ  $X_{KL}$ such that:

1. The link of $X_{KL}$  at the origin is isotopic to $L$.

2. For a fixed knot $K$ all surface germs $X_{KL}$ have isotopic tangent links.
In particular, surface germs $X_{{K_1}L}$ and $X_{{K_2}L}$ are ambient semialgebraic bi-Lipschitz
equivalent only if the knots $K_1$ and $K_2$ are isotopic.
\end{theorem}

\begin{proof} We use the construction from the proof of Theorem \ref{universality}.
Consider a characteristic band $F_K$, the characteristic cones $\widetilde{Y}_K$ and $\widetilde{X}_K$ (see Definition \ref{characteristic-band}).
Consider the surface germ $X_K$ defined in (\ref{VKYKXK}) for the knot $K$.
Let $\gamma\subset X_K$ be an arc not tangent to the set $V_K$ defined in (\ref{VKYKXK})
(i.e., $tord(\gamma',\gamma)=1$ for any $\gamma'\subset V_K$).
Let $V(\gamma)$ be a small conical neighbourhood of $\gamma$ in  $\R^4$, such that $V(\gamma)\cap X_K$ is a H\"older triangle.
Let us embed the straight cone $Z_L$ over $L$ inside $V(\gamma)$ so that its image $\tilde Z_L$ does not intersect $X_K$, and its is ambient topologically equivalent to $L$.
Let us choose two arcs $\gamma_1$ and $\gamma_2$ in $X_K\cap V(\gamma)$, and two arcs $\gamma'_1$ and $\gamma'_2$ in
$\tilde Z \cap V(\gamma)$, satisfying the following conditions:

a. $tord(\gamma_1,\gamma_2)=tord(\gamma'_1,\gamma'_2)=1$.

b. Replacing the union of the H\"older triangles $T(\gamma_1,\gamma_2)\subset X_K$ and $T(\gamma'_1,\gamma'_2)\subset\tilde Z$
with the union of H\"older triangles $T(\gamma_1,\gamma'_1)\subset V(\gamma)$ and $T(\gamma_2,\gamma'_2)\subset V(\gamma)$, as shown in Figure \ref{pic8},
we obtain a semialgebraic set $X_{KL}$ such that $X_{KL}\cap V(\gamma)$ is conical and its link is isotopic to the connected sum of $K$ and $L$.
Note that construction of $X_{KL}$ is similar to the saddle move construction in Definition \ref{saddle}.

\begin{figure}
\centering
\includegraphics[width=5in]{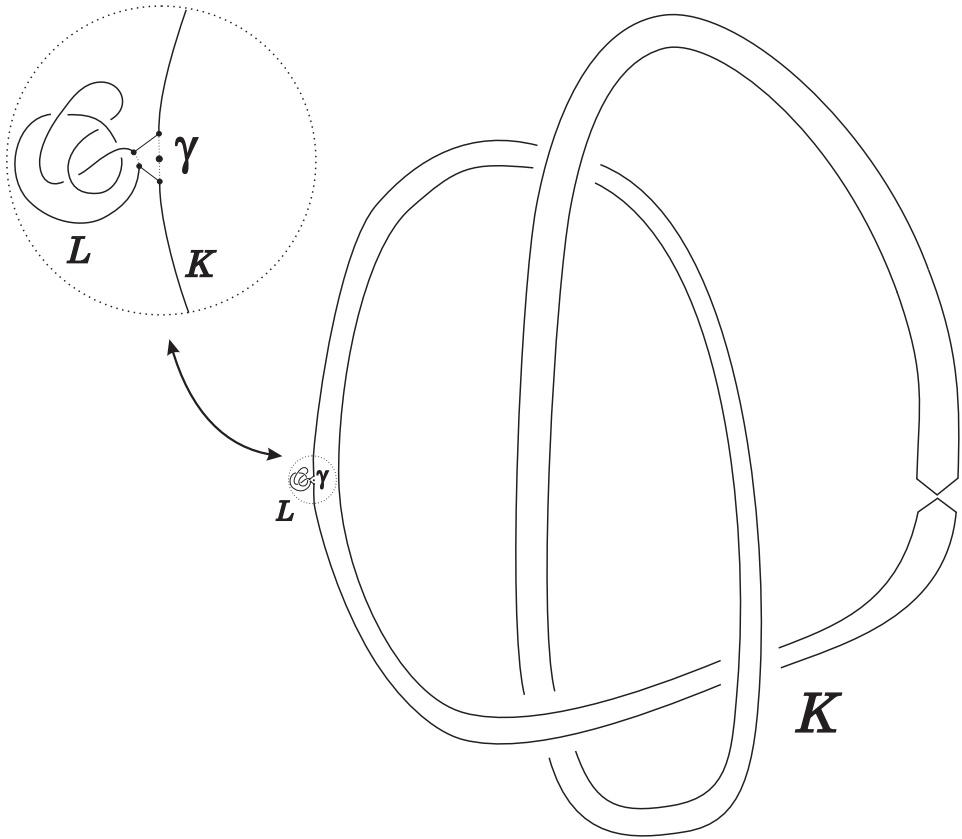}
\caption {Construction of $X_{KL}$} \label{pic8}
\end{figure}

Let us check that the surface germ $X_{KL}$ satisfies conditions of Theorem \ref{twoknots}.

1. Since the link of $X_K$ is unknotted, the connected sum is isotopic to $L$.

 The proof of the fact that, for a fixed knot $L$, all surface germs $X_{KL}$ are outer bi-Lipschitz equivalent
is the same as the proof that all surface germs $X_K$ are outer bi-Lipschitz equivalent in the proof of Theorem \ref{universality}.

2. Since $X_{KL}$ is a one-bridge surface germ, its tangent link is the union of two knots with a single common point.
One of these two knots is isotopic to $K$, and  the other one is isotopic to the connected sum of $K$ and $L$.
Since the first knot is isotopic to $K$, condition 2 is satisfied.
\end{proof}

The next result is another modification of the Universality Theorem. In contrast to the previous results, we consider surface germs with the metric structure more complicated than one-bridge.

\begin{theorem}\label{twoknots-wedge}
For any two knots $K$ and $L$, and for any two rational numbers $\alpha$ and $\beta$ such that $1\le\alpha\le\beta$,
there exists a semialgebraic surface germ $X^{\alpha\beta}_{KL}$ such that:

1. For any knots $K$ and $L$, the link at the origin of $X^{\alpha\beta}_{KL}$ is isotopic to $L$.

2. For any knots $K$ and $L$, the tangent link of $X^{\alpha\beta}_{KL}$ is isotopic to $K$.

3. For fixed $\alpha$ and $\beta$, all surface germs $X^{\alpha\beta}_{KL}$ are outer bi-Lipschitz equivalent.
\end{theorem}

\begin{proof}
Let $F_K \subset S^3$ be the characteristic band of a knot $K$ (see Definition \ref{characteristic-band}).
It is diffeomorphic to $S^1 \times [-1,1]$, and its boundary has two components $\widetilde{K}$ and $\widetilde{K}'$ isotopic to $K$.
Let $(\rho,l)$, where $\rho\in S^1$ and $l\in [-1,1]$, be coordinates in $F_K$.
Let $\widetilde{Y}_K$ and  $\widetilde{X}_K$ be the corresponding characteristic cones (see Definition \ref{characteristic-band}).
Then $(\rho,l,t)$ are coordinates in $\widetilde{Y}_K$, where $t$ is the distance to the origin.
Let $\widetilde{Y}^\alpha_K$ be a subset of $\widetilde{Y}_K$ defined as follows: \newline $\widetilde{Y}^{\alpha}_K=\{(\rho,l,t):|l|\le t^\alpha\}$.
The set $\widetilde{Y}^\alpha_K$ is called \emph{$\alpha$-contraction} of $\widetilde{Y}_K$.
Notice that the tangent link of $\widetilde{Y}^\alpha_K$ is a knot isotopic to $K$.

Let $S_K=\{(\rho,l):|\rho-\rho_0|\le\epsilon\}$ be a slice of $F_K$ (see Definition \ref{slice}) for a small $\epsilon>0$, and
let $M_K$ be the cone over $S_K$.
Let $M^\alpha_K=\{(\rho,l,t):\rho_0-\epsilon\le\rho\le\rho_0+\epsilon,\;|l|\le t^\alpha\}$ be $\alpha$-contraction of $M_K$.
Replacing $M^\alpha_K$ by the $(\alpha,\beta)$-wedge $W^{\alpha\beta}$ (see Definition \ref{wedge} and Figure \ref{fig:wt}b) as in the proof of Theorem \ref{universality}, we get the set $Y^{\alpha\beta}_K$. Let $X^{\alpha\beta}_K$ be the boundary of $Y^{\alpha\beta}_K$.

Let $\gamma\subset X ^{\alpha\beta}_K$ be an arc far from the set $W^{\alpha\beta}$,
i.e., $tord(\gamma,\gamma')=1$ for any arc $\gamma'\subset W^{\alpha\beta}$. Let $Z_L$ be the straight cone over $L$.
Let $V_\beta(\gamma)\subset\R^4$ be a $\beta$-horn like neighbourhood of $\gamma$.
Let $Z^\beta_{L,\gamma}=\Psi_\beta(Z_L)\subset V^\beta(\gamma)$ be a $\beta$-hornification of $Z_L$ to $\gamma$
(see Definition \ref{horn} and Figure \ref{horn-curve}).
Let us choose two arcs $\gamma_1$ and $\gamma_2$ in $X_K\cap V_\beta (\gamma)$, and two arcs $\gamma'_1$ and $\gamma'_2$ in
$Z^\beta_{L,\gamma} \cap V_\beta (\gamma)$ satisfying the following conditions:

a. $tord(\gamma_1,\gamma_2)=\beta,\; tord(\gamma'_1,\gamma'_2)=\beta$.

b. If we remove from $X_K$ the H\"older triangle bounded by the arcs $\gamma_1$ and $\gamma_2$, remove from $Z^\beta_{L,\gamma}$ the H\"older triangle bounded by the arcs $\gamma'_1$ and $\gamma'_2$, and add to the set $X_K \cup \tilde Z$ the H\"older triangle obtained as the union of line segments connecting $\gamma_1(t)$ and $\gamma'_1(t)$, and the H\"older triangle obtained as the union of line segments connecting $\gamma_1(t)$ and $\gamma'_2(t)$, we obtain a semialgebraic set $X^{\alpha\beta}_{KL}$ with the link isotopic to the connected sum of the links of $X_K$ and $Z^\beta_{L,\gamma}$ (see Figure \ref{pic8}).
Note that construction of $X^{\alpha\beta}_{KL}$ is similar to construction of $X_{KL}$ in the proof of Theorem \ref{twoknots} and to the saddle move construction in Definition \ref{saddle}.

Let us check that the surface germ $X^{\alpha\beta}_{KL}$ satisfies conditions of Theorem \ref{twoknots-wedge}.

1. Since $X^{\alpha\beta}_K$ has a trivial link, the connected sum is isotopic to $L$.

2. Since $Z_L$ is a subset of a $\beta$-horn neighbourhood of $\gamma$, it corresponds to a single point in the tangent link.
Thus the tangent link of $X^{\alpha\beta}_{KL}$ is the same as the tangent link of $X^{\alpha\beta}_K$, which is isotopic to $K$.

3. The proof of the fact that the surface germs $X^{\alpha\beta}_{KL}$ are outer bi-Lipschitz equivalent for a fixed $L$
is the same as the proof that all surfaces $X_{KL}$ are outer bi-Lipschitz equivalent in the proof of Theorem \ref{twoknots}.
\end{proof}

\section{Knot Invariants}\label{invariants}
\vspace{0,5cm}

In this section we make slight changes of notations. In the previous sections we used the notation $L_X$ for the link at the origin of a surface germ $X$.
Here we are going to use the notation $K_X$ if the link at the origin of $X$ is a knot, and $L_X$ if it is a topological link with more than one component.

Let us first recall the definition of the Jones polynomial $J(L)$ of a link $L$
via Kauffman bracket polynomial $\left\langle D_L\right\rangle$, where $D_L$ is a link diagram of $L$.
Kauffman bracket polynomial \cite{Ka} is a polynomial in a variable $A$
which is uniquely determined by the following properties:

\begin{enumerate}
\item Kauffman bracket on the trivial diagram equals one, i.e., $\left\langle O\right\rangle=1$
\item Skein relation
$\left\langle\ris{-4}{-1.1}{10}{L-bracket}{-1.1}\right\rangle=A\left\langle\ris{-4}{-1.1}{10}{smoothing}{-1.1}\right\rangle
+A^{-1}\left\langle\ris{-4}{-1.1}{10}{smoothing-rev}{-1.1}\right\rangle$
\item For any link diagram $D_{L'}$ we have $\left\langle O\cup D_{L'}\right\rangle=(-A^2-A^{-2})\left\langle D_{L'}\right\rangle$
\end{enumerate}
The Jones polynomial of an oriented link $L$ can be defined as
$$J(L)=(-A^3)^{-\omega(D_L)}\left\langle D_L\right\rangle,$$
after the substitution $A=t^{-\frac{1}{4}}$.
Here $\omega(D_L)$ is the writhe number of the diagram $D_L$, i.e., the number of positive
crossings minus the number of negative crossings in $D_L$.

\begin{proposition}\label{T3}
Let $X$ be a one bridge surface such that the link of $X$ at the origin is a knot $K_X$.
Let $K_{C(X)}$ be the knot, obtained from $K(X)$ by the crossing move.
Let $Y$ be a one-bridge germ such that the link at the origin of $Y$ is the same knot $K_Y=K_X$ as the link at the origin of $X$.
Let $S(Y)$ be a germ obtained from $Y$ by the saddle move. Suppose that $Y$ is such that the link at the origin of the surface $S(Y)$
is a $2$-component link $L_{S(Y)}$.
If the Jones polynomial $J(K_{C(X)})$ of the knot $K_{C(X)}$ satisfies
$$
J(K_{C(X)})\neq -t^{\frac{1}{2}}J(L_{S(Y)})+\left(-1\right)^{\omega(D_{K'})-\omega(D_K)}t^{\frac{3(\omega(D_{K'})-\omega(D_K))+1}{4}}J(K),
$$
where $D_K$ is a diagram of a knot $K$ determined by $X$, and $D_{K'}$ is a diagram (determined by the crossing move) of a knot $K_{C(X)}$
then $X$ and $Y$ are not semialgebraic ambient bi-Lipschitz equivalent.

\end{proposition}

\begin{proof}
Let $D_K$ be a diagram of a knot $K(X)$ determined by $X$ and let $D_{K'}$ be a diagram of a knot $K_{C(X)}$. Let us orient $D_K$ in an arbitrary way. We orient $D_{K'}$ so that
the intersection, corresponding to the crossing move (see Figure \ref{s-c}) on the  diagram is positive, i.e., it looks like
$\left(\ris{-4}{-1}{10}{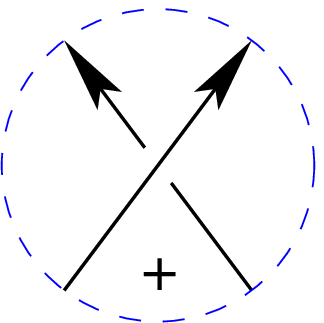}{-1.1}\right)$.
Let $S(X)$ be a germ of a surface obtained from $X$ by a saddle move. Let $D_L$ be the
corresponding diagram of the characteristic link $L_{S(X)}$. We orient $D_L$ such that
the part, corresponding to the saddle move (see Figure \ref{s-c}) looks like $\left(\ris{-4}{-1.1}{10}{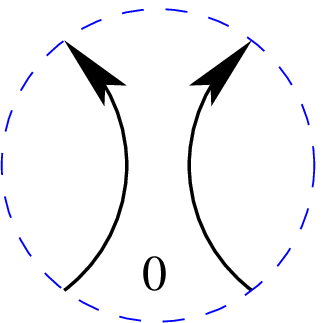}{-1.1}\right)$.
Before the substitution $A=t^{-\frac{1}{4}}$ we have
$$
\left\langle D_{K'}\right\rangle=(-A^3)^{\omega(D_{K'})}J(K_{C(X)})\qquad
\left\langle D_L\right\rangle=(-A^3)^{\omega(D_L)}J(L_{S(X)}).
$$
Now it follows from the condition (2) of the Kauffman bracket that
$$
(-A^3)^{\omega(D_{K'})}J(K_{C(X)})=A(-A^3)^{\omega(D_L)}J(L_{S(X)})+A^{-1}(-A^3)^{\omega(D_K)}J(K).
$$
Using the fact that $\omega(D_{K'})=\omega(D_L)+1$ and after the substitution $A=t^{-\frac{1}{4}}$ we get
\begin{equation}\label{eq:Jones1}
J(K_{C(X)})=-t^{\frac{1}{2}}J(L_{S(X)})+\left(-1\right)^{\omega(D_{K'})-\omega(D_K)}t^{\frac{3(\omega(D_{K'})-\omega(D_K))+1}{4}}J(K).
\end{equation}

Recall that Proposition \ref{saddle-saddle} implies that if the link
$L_{S(X)}$ is not isotopic to the link $L_{S(Y)}$,
then $X$ and $Y$ are not semi-algebraic ambient bi-Lipschitz equivalent. Hence if
$J(L_{S(X)})\neq J(L_{S(Y)})$, then $X$ and $Y$ are not semialgebraic ambient bi-Lipschitz equivalent.
Now equality \eqref{eq:Jones1} yields the proof of the proposition.
\end{proof}

\begin{figure}
\centering
\includegraphics[width=5in]{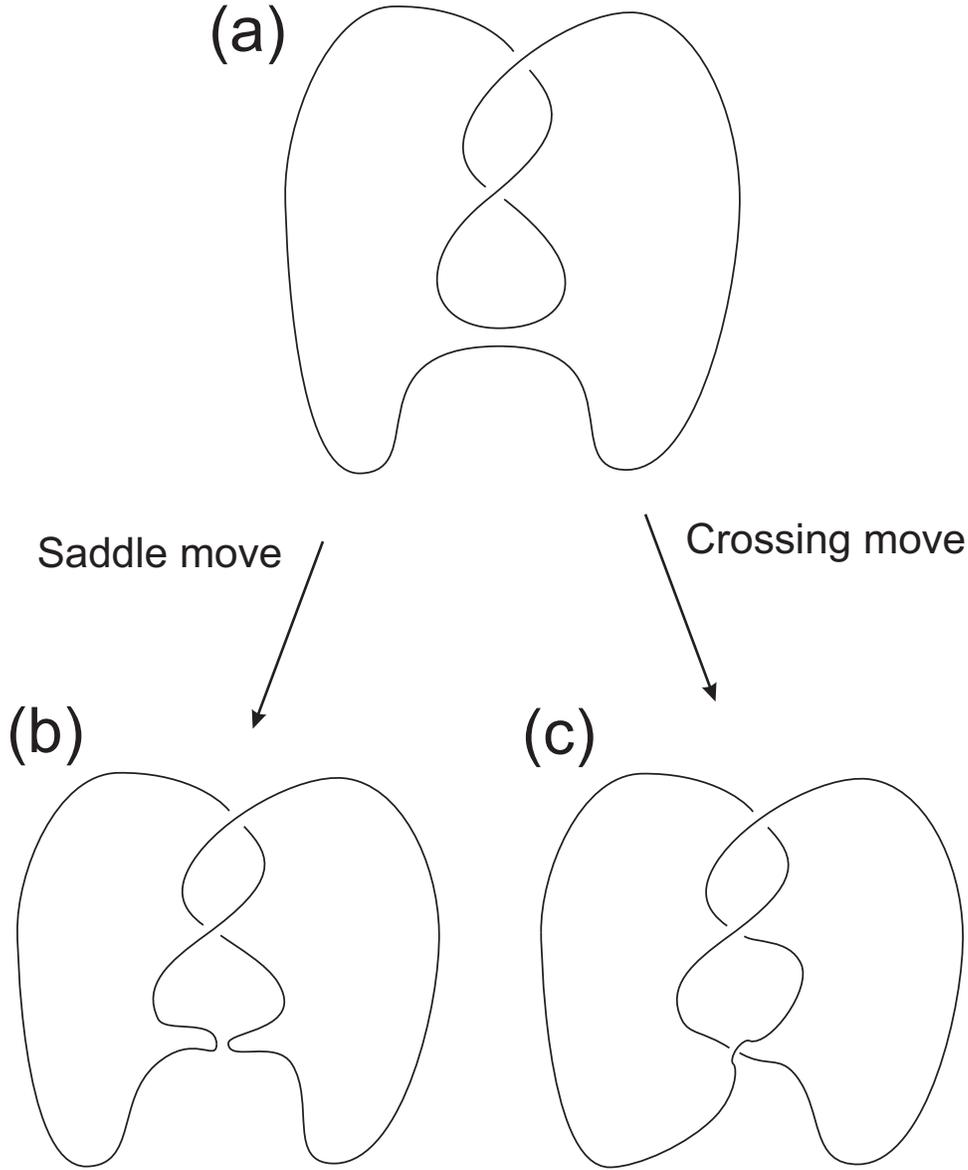}
\caption {The Saddle move and the Crossing move} \label{s-c}
\end{figure}

\begin{corollary}\label{specialcase}
 If $K$ is a trivial knot and $L_{S(Y)}$ is $(2,2m)$-torus link $L(2,2m)$, where $m$ is a non-negative integer,
then we get the following closed formula: If the Jones polynomial $J(K_{C(X)})$ of the knot $K_{C(X)}$ satisfies
\begin{equation}\label{eq:torus-links}
J(K_{C(X)})\neq t^m+t^{m+2}\left(\frac{1+t^{2m-1}}{1+t}\right)+\left(-1\right)^{\omega(D_{K'})-\omega(D_K)}t^{\frac{3(\omega(D_{K'})-\omega(D_K))+1}{4}},
\end{equation}
then $X$ and $Y$ are not semialgebraic ambient bi-Lipschitz equivalent.
\end{corollary}

\begin{proof}
Recall that for each $n$ the Jones polynomial of the torus
knot $K(2,2n+1)$ equals
$$
J(K(2,2n+1))=t^n\frac{1-t^3-t^{2n+2}+t^{2n+3}}{1-t^2},
$$
see e.g. \cite{Jones}. The skein relation for the Jones polynomial together with the above equality yield
\begin{equation}\label{eq:Jones2}
J(L(2,2m))=-t^{\frac{2m-1}{2}}-t^{\frac{2m+3}{2}}\left(\frac{1+t^{2m-1}}{1+t}\right).
\end{equation}

Noting that if $K$ is a trivial knot, then its Jones polynomial $J(K)=1$, and applying
equalities \eqref{eq:Jones1} and \eqref{eq:Jones2} we obtain the proof of the corollary.
\end{proof}


\begin{remark}\label{last-remark} \emph{
The above theorem has two advantages: it has a computational value, and as
its immediate corollary we obtain the main result of Birbrair-Gabrielov \cite[Theorem 4.1]{BG}.
Let us illustrate this on the following example. Let $X$ be such that it determines a knot diagram $D_K$ which has no intersections,
and after the crossing move the diagram $D_{K'}$ has exactly one positive intersection. It follows that $\omega (D_{K'})-\omega (D_K)=1$,
and $J(K_{C(X)})=1$ since $K_{C(X)}$ is a trivial knot.
Let $Y$ be such that it determines a trivial knot diagram presented in Figure \ref{s-c}a. The diagram of the link $L_{S(Y)}$ is
presented in Figure \ref{s-c}b. Note that it is a $(2,2)$-torus link (Hopf link). The diagram of the knot $K_{C(Y)}$
is presented in Figure \ref{s-c}c. Note that it is a trefoil knot. Noting that $m=1$ the right hand side
of equation \eqref{eq:torus-links} equals $t^3$. Hence $J(K_{C(X)})\neq t^3$ and thus $X$ and $Y$ are not semialgebraic ambient bi-Lipschitz equivalent.}
\end{remark}

\end{document}